\documentclass{amsart}

\usepackage{amssymb,amsmath,amsthm,latexsym,stmaryrd,xcolor}

 \newtheorem{thm}{Theorem}[section]
 \newtheorem{cor}[thm]{Corollary}

 \newtheorem{prop}[thm]{Proposition}
 \theoremstyle{definition}
 
 \theoremstyle{remark}

\newcommand{\ie}{i.e.\ }
\newcommand{\f}{\phi}
\newcommand{\tg}{\tilde{g}}
\newcommand{\n}{\nabla}

\newcommand{\M}{(\mathcal{M},\allowbreak{}\f,\allowbreak{}\xi,\allowbreak{}\eta,\allowbreak{}g)}

\newcommand{\F}{\mathcal{F}}
\newcommand{\I}{\mathcal{I}}

\newcommand{\MM}{\mathcal{M}}
\newcommand{\LL}{\mathfrak{L}}
\newcommand{\LLL}{\mathcal{L}}
\newcommand{\HH}{\mathcal{H}}
\newcommand{\R}{\mathbb R}

\newcommand{\X}{\mathfrak X}

\newcommand{\tr}{{\rm tr}}
\newcommand{\ta}{\theta}
\newcommand{\om}{\omega}
\newcommand{\lm}{\lambda}

\newcommand{\al}{\alpha}
\newcommand{\bt}{\beta}

\newcommand{\dsfrac}{\displaystyle\frac}

\newcommand{\ad}{\operatorname{Ad}}

\newcommand{\thmref}[1]{Theorem~\ref{#1}}

\newcommand{\corref}[1]{Corollary~\ref{#1}}
\newcommand{\propref}[1]{Proposition~\ref{#1}}

\begin{document}

\title[Lie Groups as 3-Dimensional Almost Paracontact \ldots]
{Lie Groups as 3-Dimensional Almost Paracontact Almost
Paracomplex Riemannian  
Manifolds}

\author[M. Manev]{Mancho Manev}
\address[M. Manev]{University of Plovdiv Paisii Hilendarski,   Faculty of Mathematics and
Informatics,   Department of Algebra and Geometry,   24 Tzar Asen St.,
4000 Plovdiv,  Bulgaria;}
\address[M. Manev]{Medical University of Plovdiv, Faculty of Public
Health, Department of Medical Informatics, Biostatistics and E-Learning, 15-A Vasil Aprilov
Blvd., 4002  Plovdiv,   Bulgaria}
\email{mmanev@uni-plovdiv.bg}

\author[V. Tavkova]{Veselina Tavkova}
\address[V. Tavkova]{University of Plovdiv Paisii
Hilendarski,
Faculty of Mathematics and Informatics,   Department of Algebra
and Geometry,   24 Tzar Asen St.,
4000 Plovdiv, Bulgaria} \email{vtavkova@uni-plovdiv.bg}





\begin{abstract}
Almost paracontact almost paracomplex Riemannian manifolds of the lowest dimension 3
are considered.
Such structures are constructed on a family of Lie groups and the obtained manifolds are studied.
Curvature properties of these manifolds are investigated. An example is commented as support of obtained results.
\end{abstract}

\maketitle

\section*{Introduction}

The object of our considerations is geometry of the so-called almost paracontact almost paracomplex Riemannian manifolds. The restriction of the introduced almost paracontact structure on the paracontact distribution is an almost paracomplex structure. The more popular case is when the compatible metric with the almost paracontact structure is  Riemannian, although the metric can be also indefinite.

Most generally, the notion of almost paracontact structure on a differentiable manifold of arbitrary dimension was introduced by I. Sato \cite{Sato76}.
The restriction of this structure on the paracontact distribution is an almost product structure classified by A.M. Naveira \cite{Nav}.

The almost paracontact structure is an analogue of almost contact structure although almost contact manifolds are necessarily odd-dimensional whereas almost paracontact manifolds could be even-dimensional as well.

More close analogue of an almost complex structure for the considered manifolds is the case when the induced almost product structure is traceless. Then such a structure is called almost paracomplex structure.
These manifolds are called almost paracontact almost paracomplex manifolds \cite{ManVes18}.
They have dimension $(2n+1)$ and are classified under the name of almost paracontact
Riemannian manifolds of type $(n, n)$ by M. Manev and M. Staikova in
\cite{ManSta01}.

A number of authors have studied Lie groups as manifolds equipped with various tensor structures and metrics that are compatible with the structures (including in the lowest-dimensional cases) --  for example, \cite{Blair} and \cite{BBV} for almost contact metric manifolds, \cite{HManMek} and \cite{MIHM} for almost contact B-metric manifolds, \cite{AbbGarb} and \cite{FCG} for almost complex manifolds with Hermitian metric,  \cite{GriManMek} and \cite{MT} for almost complex manifolds with Norden metric, \cite{Barb} and \cite{DF} for hypercomplex hyper-Hermitian manifolds, \cite{GriMan} and \cite{Man44} for almost hypercomplex Hermitian-Norden manifolds, \cite{DobrMek} and \cite{Sht} for Riemannian almost product manifolds, \cite{VMMolina} and \cite{ZamNak} for almost paracontact metric manifolds.

The goal of
the present work is to study the geometric characteristics and properties of
a family of Lie groups considered as 3-dimensional almost paracontact almost paracomplex Riemannian manifolds.
The expected results will provide  a series of explicit examples of the manifolds studied and will contribute to understanding of their geometry.

The  paper is organized as follows. In
Sect.~\ref{sect-prel} we give some preliminary facts and definitions for the studied manifolds.
In Sect.~\ref{sect-Lie-mfds} we construct and characterize a family of 3-dimensional Lie groups considered  as almost paracontact almost paracomplex Riemannian manifolds.
In Sect.~\ref{sect-Exms} we give an example in relation with
the  above investigations.

\section{Preliminaries}\label{sect-prel}

\subsection{Structures of almost paracontact almost paracomplex Riemannian manifolds}\label{sect-mfds}

Let $(\MM,\f,\xi,\eta)$ be an almost paracontact almost paracomplex manifold, \ie $\MM$
is a $(2n+1)$-dimensional differentiable manifold with an almost
paracontact structure consisting of a tensor field $\f$ of type  $(1,1)$ on the tangent bundle $T\MM$ of $\MM$, a vector field $\xi$ and an 1-form
 $\eta$, satisfying the following conditions:
\begin{equation}\label{str}
\begin{array}{c}
\f^2 = \I - \eta \otimes \xi,\quad \eta(\xi)=1,\quad
\eta\circ\f=0,\quad \f\xi = 0,\quad \tr\f = 0,
\end{array}
\end{equation}
where $\I$ is the identity on $T\MM$ \cite{Sato76}.

Moreover, $(\MM,\f,\xi,\eta)$ admits a Riemannian metric $g$ which is compatible with the structure of the manifold by the following way:
\begin{equation}\label{metric}
g(\f x, \f y) = g(x,y) - \eta(x)\eta(y),\qquad g(x,\xi) = \eta(x).
\end{equation}
Then $(\MM,\f,\xi,\eta,g)$ is called an \emph{almost paracontact almost paracomplex Riemannian manifold} \cite{ManVes18}.

Here and further $x$, $y$, $z$, $w$ will stand for arbitrary
elements of the Lie algebra $\X(\MM)$ of tangent vector fields on $\MM$ or vectors in the tangent space $T_p\MM$ at $p\in \MM$.

Let us recall that the endomorphism $\f$ induces an almost paracomplex structure on each fibre of the $2n$-dimensional paracontact distribution $\HH=\ker(\eta)$ of $T\MM$.
Furthermore, an almost paracomplex structure is an almost product structure  $P$ (\ie $P^2=\I$ and  $P\neq \pm \I$) such that the eigenvalues $+1$ and $-1$ of $P$ have one and the same multiplicity $n$, \ie $\tr P=0$ follows.

The associated metric $\tg$ of $g$ on $(\MM, \f, \xi,\eta, g)$ is defined by the equality \(\tg(x,y)=g(x,\f y)+\eta(x)\eta(y)\). Obviously, it is a compatible metric for $(\MM, \f, \allowbreak{}\xi,\allowbreak{}\eta)$, \ie relations \eqref{metric} are valid for $\tg$ and $(\f, \xi,\eta)$, as well as it is a pseudo-Riemannian metric of signature  $(n+1,n)$. Therefore,  $(\MM, \f, \xi,\eta,\tg)$ is an almost paracontact almost paracomplex pseudo-Riemannian manifold.

According to \cite{ManVes18}, the decomposition $x=\f^2x+\eta(x)\xi$ due to  (\ref{str}) generates the projectors
$h$ and $v$ on any tangent space $T_p\MM$ of $(\MM,\f,\xi,\eta,g)$,
determined by $hx=\f^2x$ and $vx=\eta(x)\xi$. Then, it is obtained the
orthogonal decomposition $T_p\MM=h(T_p\MM)\oplus v(T_p\MM)$. Moreover, it generates
the corresponding orthogonal decomposition of the space $\mathcal{S}$ of the
$(0,2)$-tensors $S$ over $(\MM,\f,\xi,\eta,g)$ as follows:
\[
\mathcal{S}=\ell_1(\mathcal{S})\oplus \ell_2(\mathcal{S})
\oplus \ell_3(\mathcal{S}),
\qquad
\ell_i(\mathcal{S})=\left\{S\in \mathcal{S}\ |\
S=\ell_i(S)\right\},\quad i=1,2,3;
\]
\begin{equation*}\label{ell}
\begin{array}{l}
\ell_1(S)(x,y)=S(hx,hy),\qquad \ell_2(S)(x,y)=S(vx, vy),\\[4pt]
\ell_3(S)(x,y)=S(vx, hy)+S(hx, vy).
\end{array}
\end{equation*}

Thus, for $g, \tg\in \mathcal{S}$, we have:
\begin{equation}\label{lg}
	\begin{array}{lll}
	\ell_1(g)=g(\f\cdot,\f\cdot)=g-\eta\otimes\eta, \quad
	&\ell_2(g)=\eta\otimes\eta,
	\quad &\ell_3(g)=0,\\[0pt]
	\ell_1(\tg)=g(\cdot,\f\cdot)=\tg-\eta\otimes\eta, \quad
	&\ell_2(\tg)=\eta\otimes\eta, \quad &\ell_3(\tg)=0.
	\end{array}
\end{equation}

\subsection{Curvatures of the considered manifolds}\label{sec-curv}

The curvature tensor $R$ of type $(0,3)$ for the Levi-Civita connection $\nabla$ of $g$ is determined as usually by $R=\left[\n,\n\right]-\n_{[\ ,\ ]}$.
The corresponding $(0,4)$-tensor is denoted by the same letter and it is defined by  $R(x,y,z,w)=g(R(x,y)z,w)$.  With respect to an arbitrary basis, the Ricci tensor $\rho$ and the scalar curvature $\tau$ for $R$ as well as
their associated quantities are determined by:
\begin{equation*}
\begin{array}{ll}
    \rho(y,z)=g^{ij}R(e_i,y,z,e_j),\qquad &
    \tau=g^{ij}\rho(e_i,e_j),\\[4pt]
    \rho^*(y,z)=g^{ij}R(e_i,y,z,\f e_j),\qquad &
    \tau^*=g^{ij}\rho^*(e_i,e_j).
\end{array}
\end{equation*}


Further, we use the Kulkarni-Nomizu
product $g\owedge  h$ of two (0,2)-tensors $g$ and $h$ defined by
\[
\begin{array}{l}
\left(g\owedge h\right)(x,y,z,w)=g(x,z)h(y,w)-g(y,z)h(x,w)\\
\phantom{\left(g\owedge h\right)(x,y,z,w)}
+g(y,w)h(x,z)-g(x,w)h(y,z).
\end{array}
\]
Moreover,   $g\owedge h$ has the basic properties of $R$ 
if and only if $g$ and $h$ are symmetric.

Let $\al$ be a non-degenerate 2-plane in $T_p\MM$, $p \in \MM$, having a basis  $\{x,y\}$.
The sectional curvature $k(\al;p)$ with respect to $g$ and $R$ is determined by
\begin{equation*}\label{sect}
k(\al;p)=-\frac{2R(x,y,y,x)}{(g\owedge g)(x,y,y,x)}.
\end{equation*}

It is known that a 2-plane is called a \emph{$\f$-holomorphic section}
(respectively, a \emph{$\xi$-sec\-tion})
if $\al= \f\al$ (respectively, $\xi \in \al$).

Let us recall that on each 3-dimensional manifold the curvature tensor has the following form:
\begin{equation}\label{R3}
R 
=-g\owedge \left(\rho-\frac{\tau}{4} g\right).
\end{equation}

%
%
%

As it is known,  a manifold is called \emph{Einstein} if the
Ricci tensor is proportional to the metric tensor, \ie $\rho=\lm
g$, $\lm\in\R$.

For the manifolds studied $\M$, besides the metric $g$, we also have its associated metric $\tg$ and their component $\eta\otimes\eta$ according to \eqref{lg}. Then, it is reasonable to consider the following more general case of the Einstein property, similarly to \cite{HManMek} for almost contact B-metric manifolds.
An almost paracontact almost paracomplex Riemannian manifold is called \emph{$\eta$-paracomplex-Einstein} when the following condition is valid
\begin{equation*}
\rho=\lm g + \mu \tg + \nu \eta\otimes\eta, \qquad \lm,\mu,\nu\in\R.
\end{equation*}
In particular, if $\mu = 0$ then $\MM$ is called \emph{para-$\eta$-Einstein}.

For almost paracontact metric manifolds, there is no an associated metric $\tg$ and thus the para-$\eta$-Einstein kind is only applicable. In this regard, several authors consider
Sasakian and paracontact metric manifolds satisfying the para-$\eta$-Einstein condition and corresponding properties  are well studied, e.g.
\cite{Ok62}, \cite{SiSha83}, \cite{Parch16}.

In the present paper, we consider also the Einstein condition for the manifolds $\M$ regarding the separate components of the metrics, 
according to \eqref{lg}.

A para-$\eta$-Einstein
manifold is said to be  an \emph{$\ell_i$-para-$\eta$-Einstein
manifold} ($i\in\{1,2\}$) when the condition $\rho=\lm\, \ell_i(g)$, $\lm\in\R$, is satisfied.
Similarly, there are meaningful respective notions regarding $\tg$.

\subsection{Basic classes of the considered manifolds}\label{sec-class}

In \cite{ManSta01}, it is given a classification of almost paracontact almost paracomplex Riemannian manifolds consisting of eleven basic classes  $\F_1$, $\F_2$, $\dots$, $\F_{11}$. It is made with respect to the tensor $F$ of type  (0,3) defined by
\begin{equation*}\label{F=nfi}
F(x,y,z)=g\bigl( \left( \nabla_x \f \right)y,z\bigr).
\end{equation*}
The basic properties of $F$ with respect to the structure are the following:
\begin{equation*}\label{F-prop}
F(x,y,z)=F(x,z,y)=-F(x,\f y,\f z)+\eta(y)F(x,\xi,z)
+\eta(z)F(x,y,\xi).
\end{equation*}

Let $\left\{\xi; e_i\right\}$ $(i=1,2,\dots,2n)$ is a basis of the tangent space
$T_p\MM$ at an arbitrary point $p\in \MM$. The components of the inverse matrix of $g$ are denoted by $g^{ij}$, then the Lee 1-forms $\theta$, $\theta^*$, $\omega$ associated with $F$ are defined by:
\begin{equation*}\label{t}
\theta(z)=g^{ij}F(e_i,e_j,z),\quad \theta^*(z)=g^{ij}F(e_i,\f
e_j,z), \quad \omega(z)=F(\xi,\xi,z).
\end{equation*}

The intersection of the basic classes is the special class $\F_0$
determined by the condition $F(x,y,z)=0$ and it is known as the class with $\n$-parallel
structures, \ie $\n\f=\n\xi=\n\eta=\n g=\n \tg=0$.

In \cite{ManSta01}, there are given the conditions for $F$ determining the basic classes $\F_{i}$ of $\M$, whereas the components $F_{i}$ of $F$ corresponding to $\F_{i}$ are known from \cite{ManVes18}.
Namely, the manifold $\M$ belongs to $\F_{i}$ $(i\in\{1,2,\dots,11\})$
if and only if the equality $F=F_i$ is valid. In the latter case, $\M$ is also called an $\F_{i}$-manifold.

Moreover, a studied manifold
belongs to a direct sum of two or more basic classes, \ie
$\M\in\F_i\oplus\F_j\oplus\cdots$, if and only if $F$ is the sum of the corresponding components
$F_i$, $F_j$, $\ldots$, \ie the following condition is
satisfied $F=F_i+F_j+\cdots$.

In the present paper, we consider the case of the lowest dimension of the manifolds under study, \ie $\dim{\MM}=3$.

Then, the basic  classes of the 3-dimensional manifolds of the considered type are
$\F_1$, $ \F_4$, $\F_5$,  $\F_8$, $\F_9$, $\F_{10}$,  $\F_{11}$, \ie $\F_2$, $\F_3$, $\F_6$, $\F_7$ are restricted to $\F_0$ \cite{ManVes18}.

Let $\left\{e_0=\xi,e_1=e,e_2=\f e\right\}$ be a \emph{$\f$-basis} of $T_p\MM$, therefore it is an orthonormal basis  with respect to $g$, \ie $g(e_i,e_j)=\delta_{ij}$ for all $i,j\in\{0,1,2\}$.
We denote the components of $F$, $\ta$, $\ta^*$ and $\om$  with respect to this $\f$-basis 
as follows
\[
{F_{ijk}=F(e_i,e_j,e_k)},\quad {\ta_k=\ta(e_k)},\quad {\ta^*_k=\ta^*(e_k)},\quad {\om_k=\om(e_k)}.
\]

In the final part of the present section we recall the needed results from \cite{ManVes18}.

The components of the Lee forms with respect to the $\f$-basis are:
\begin{equation}\label{t3}
\begin{array}{c}
	\begin{array}{ll}
		\ta_0=F_{110}+F_{220},\quad & \ta_1=F_{111}=-F_{122}=-\ta^*_2,\\[0pt]
		\ta^*_0=F_{120}+F_{210}, \quad &\ta_2=F_{222}=-F_{211}=-\ta^*_1,\\[0pt]
	\end{array}\\
	\begin{array}{lll}
		\om_0=0,  \qquad & \om_1=F_{001},\qquad & \om_2=F_{002}.
	\end{array}
\end{array}
\end{equation}

Further, if $F_s$ $(s=1,2,\dots,11)$ are the components of  $F$ in the corresponding basic classes $\F_s$, we have:
\begin{equation}\label{Fi3}
\begin{array}{l}
F_{1}(x,y,z)=\left(x^1\ta_1-x^2\ta_2\right)\left(y^1z^1-y^2z^2\right); \\[0pt]
F_{2}(x,y,z)=F_{3}(x,y,z)=0;
\\
F_{4}(x,y,z)=\frac{\ta_0}{2}\Bigl\{x^1\left(y^0z^1+y^1z^0\right)
+x^2\left(y^0z^2+y^2z^0\right)\bigr\};\\[0pt]
F_{5}(x,y,z)=\frac{\ta^*_0}{2}\bigl\{x^1\left(y^0z^2+y^2z^0\right)
+x^2\left(y^0z^1+y^1z^0\right)\bigr\};\\[0pt]
F_{6}(x,y,z)=F_{7}(x,y,z)=0;\\[0pt]
F_{8}(x,y,z)=\lm\bigl\{x^1\left(y^0z^1+y^1z^0\right)
-x^2\left(y^0z^2+y^2z^0\right)\bigr\},\\[0pt]
\hspace{38pt} \lm=F_{110}=-F_{220}
;\\[0pt]
F_{9}(x,y,z)=\mu\bigl\{x^1\left(y^0z^2+y^2z^0\right)
-x^2\left(y^0z^1+y^1z^0\right)\bigr\},\\[0pt]
\hspace{38pt} \mu=F_{120}=-F_{210}
;\\[0pt]
F_{10}(x,y,z)=\nu x^0\left(y^1z^1-y^2z^2\right),\quad
\nu=F_{011}=-F_{022}
;\\[0pt]
F_{11}(x,y,z)=x^0\bigl\{\om_{1}\left(y^0z^1+y^1z^0\right)
+\om_{2}\left(y^0z^2+y^2z^0\right)\bigr\},
\end{array}
\end{equation}
where $x=x^ie_i$, $y=y^ie_i$, $z=z^ie_i$ are arbitrary vectors in $T_p\MM$, $p\in \MM$.

\section{Lie groups as 3-dimensional manifolds of the studied type}\label{sect-Lie-mfds}

Let $\LLL$ be a 3-dimensional real connected Lie group and
$\mathfrak{l}$ be its Lie algebra. If
$\{E_{0},E_{1},E_{2}\}$ is a basis of left invariant vector fields on $\mathfrak{l}$ then an almost paracontact almost paracomplex structure $(\f,\xi,\eta)$ and a Riemannian metric $g$ can be determined by the following way:
\begin{equation}\label{strL}
\begin{array}{l}
\f E_0=0,\quad \f E_1=E_{2},\quad \f E_{2}= E_1,\quad \xi=
E_0,\quad \\[4pt]
\eta(E_0)=1,\quad \eta(E_1)=\eta(E_{2})=0,
\end{array}
\end{equation}
\begin{equation}\label{gL}
  g(E_i,E_j)=\delta_{ij},\qquad i,j\in\{0,1,2\}.
\end{equation}
Thus, we obtain the manifold $(\LLL,\f,\xi,\eta,g)$. Obviously, we have the following
\begin{prop}
The manifold $(\LLL,\f,\xi,\eta,g)$ is a 3-dimensional almost paracontact almost paracomplex Riemannian manifold.
\end{prop}

Further, the denotation $(\LLL,\f,\xi,\eta,g)$ stands for this manifold.

The corresponding Lie algebra $\mathfrak{l}$
 is defined by:
\begin{equation}\label{lie}
\left[E_{i},E_{j}\right]=C_{ij}^k E_{k}, \quad i, j, k \in \{0,1,2\}.
\end{equation}

From the nine commutation coefficients $C_{ij}^k$, using the Jacobi identity
\[
\bigl[[E_i,E_j],E_k\bigr]+\bigl[[E_j,E_k],E_i\bigr]+\bigl[[E_k,E_i],E_j\bigr]=0,
\]
remain six which
could be chosen as parameters. So, we express the three coefficients with
different indices by the six parameters (if the
denominators are non-zero) as follows:
\begin{equation*}\label{Jac}
\begin{array}{c}
C_{12}^0=\dsfrac{ C_{01}^0 C_{12}^1+ C_{02}^0 C_{12}^2}{C_{01}^1 +
C_{02}^2},\qquad C_{02}^1=\dsfrac{ C_{02}^0 C_{01}^1-C_{12}^1
C_{02}^2}{C_{01}^0-C_{12}^2},
\\[4pt]
C_{01}^2=\dsfrac{C_{01}^0 C_{02}^2+C_{01}^1 C_{12}^2}{C_{02}^0 +
C_{12}^1}.
\end{array}
\end{equation*}

Using the known property of the Levi-Civita connection of $g$
\begin{equation}\label{Kosz}
2g\left(\n_{E_i}E_j,E_k\right)
=g\left([E_i,E_j],E_k\right)+g\left([E_k,E_i],E_j\right)
+g\left([E_k,E_j],E_i\right),
\end{equation}
we get the following formula for the $F$'s components
$F_{ijk}=F(E_i,E_j,E_k)$, $i,j,\allowbreak{}k\in\{0,1,2\}$:
\begin{equation*}\label{Fijk}
\begin{split}
2F_{ijk}=g\left([E_i,\f E_j]-\f [E_i,E_j],E_k\right)
&+g\left(\f [E_k,E_i]-[\f E_k,E_i],E_j\right)\\[4pt]
&+g\left([E_k,\f E_j]-[\f E_k,E_j],E_i\right).
\end{split}
\end{equation*}

Then, we have the following equations:
\begin{equation}\label{FijkC}
\begin{array}{l}
\begin{array}{ll}
F_{111}=-F_{122}=2C_{12}^1,\quad &
F_{211}=-F_{222}=2C_{12}^2,\\[4pt]
F_{120}=F_{102}=C_{01}^1,\quad &
F_{020}=F_{002}=C_{01}^0,\\[4pt]
F_{210}=F_{201}=C_{02}^2,\quad &
F_{010}=F_{001}=C_{02}^0,\\[4pt]
\end{array}\ \\[4pt]
\begin{array}{l}
F_{110}=F_{101}=\frac12 \left(C_{12}^0+C_{02}^1+C_{01}^2\right),\\[4pt]
F_{220}=F_{202}=\frac12 \left(-C_{12}^0+C_{02}^1+C_{01}^2\right),\\[4pt]
F_{011}=-F_{022}=C_{12}^0+C_{02}^1-C_{01}^2,
\end{array}
\end{array}
\end{equation}
and the other components $F_{ijk}$ are zero.

Hence, we obtain the following equalities for the Lee forms:
\begin{equation}\label{titiC}
\begin{array}{l}
\begin{array}{l}
\ta_{0}=C_{02}^1+C_{01}^2,\\[4pt]
\ta_{1}=2C_{12}^1,\\[4pt]
\ta_{2}=-2C_{12}^2,\\[4pt]
\end{array}\quad
\begin{array}{l}
\ta^*_{0}=C_{01}^1+C_{02}^2,\\[4pt]
\ta^*_{1}=2C_{12}^2,\\[4pt]
\ta^*_{2}=-2C_{12}^1,\\[4pt]
\end{array}\quad
\begin{array}{l}
\om_{0}=0,\\[4pt]
\om_{1}=C_{02}^0,\\[4pt]
\om_{2}=C_{01}^0.
\end{array}
\end{array}
\end{equation}

\begin{thm}\label{thm-Fi-L}
The manifold $(\LLL,\f,\xi,\eta,g)$ belongs to the basic class $\F_s$
($s \in \{1,\allowbreak{}4,5,8,9,10,11\}$) if and only
if the corresponding Lie algebra $\mathfrak{l}$ is determined by
the following commutators:
\begin{equation*}\label{Fi-L}
\begin{array}{llll}
\F_1:\; &[E_0,E_1]=0, \; & [E_0,E_2]=0, \; &   [E_1,E_2]=\al E_1-\bt E_2;
\\[4pt]
\F_4:\; &[E_0,E_1]=\al E_2, \; &   [E_0,E_2]=\al E_1, \; &
[E_1,E_2]=0;
\\[4pt]
\F_5:\; &[E_0,E_1]=\al E_1, \; &   [E_0,E_2]=\al E_2, \; &
[E_1,E_2]=0;
\\[4pt]
\F_8:\; &[E_0,E_1]=\al E_2, \; &   [E_0,E_2]=-\al E_1, \; &
[E_1,E_2]=2\al E_0;
\\[4pt]
\F_9:\; &[E_0,E_1]=\al E_1, \; &   [E_0,E_2]=-\al E_2, \; &
[E_1,E_2]=0;
\\[4pt]
\F_{10}:\; &[E_0,E_1]=-\al E_2, \; &   [E_0,E_2]=\al E_1, \; &
[E_1,E_2]=0;
\\[4pt]
\F_{11}:\; &[E_0,E_1]=\al E_0, \; &   [E_0,E_2]=\bt E_0, \; &
[E_1,E_2]=0,
\end{array}
\end{equation*}
where $\al$, $\bt$ are arbitrary real parameters.
Moreover, the
relations of $\al$ and $\bt$ with the non-zero components
$F_{ijk}$ in the different basic classes $\F_s$ from \eqref{Fi3} are as follows:
\begin{equation*}\label{Fi-L-alpha}
\begin{array}{ll}
\F_1:\quad \al=\frac12 \ta_1,\quad \bt=-\frac12 \ta_2; \qquad
&\F_4:\quad \al=\frac12 \ta_0;\\[4pt]
\F_5:\quad \al=\frac12 \ta^*_0; \qquad
&\F_8:\quad \al=\lm; \\[4pt]
\F_9:\quad \al=\mu; \qquad
&\F_{10}:\quad \al=\frac12 \nu;   \\[4pt]
\F_{11}:\quad \al=\om_2, \quad \bt=\om_1.\qquad &
\end{array}
\end{equation*}
\end{thm}
\begin{proof}
The calculations are made, using \eqref{t3}, \eqref{Fi3}, \eqref{FijkC} and  \eqref{titiC}.
\end{proof}

Let us remark that the class of the para-Sasakian paracomplex Riemannian
manifolds is $\F'_4$, where $\F'_4$ is the subclass of $\F_4$ determined by the condition $\ta(\xi)=-2n$
\cite{ManVes18}.

Then, due to \thmref{thm-Fi-L}, we have the following
\begin{cor}
The manifold $(\LLL,\f,\xi,\eta,g)$ is para-Sasakian if and only
if the corresponding Lie algebra $\mathfrak{l}$ is determined by
the following commutators:
\[
[E_0,E_1]=-E_2, \quad   [E_0,E_2]=-E_1, \quad [E_1,E_2]=0.
\]
\end{cor}

Let us note that an $\F_{0}$-manifold is obtained if and only if the Lie
algebra is Abelian, \ie all commutators are zero. Further, we skip
this special case.

\subsection{Some special structures  on the considered manifolds}

A metric $g$ is called \emph{Killing} (or,
\emph{$\ad$-invariant}) if satisfying the property
\begin{equation}\label{g-Kil}
g([x,y],z)=g(x,[y,z]).
\end{equation}

\begin{thm}
The metric $g$ of $(\LLL,\f,\xi,\eta,g)$ is Killing if and only if $(\LLL,\allowbreak{}\f,\allowbreak{}\xi,\allowbreak{}\eta,\allowbreak{}g)$
belongs to the subclass of $\F_8\oplus\F_{10}$ determined by the
condition $2\lm=-\nu$.
\end{thm}
\begin{proof}
According to \eqref{gL}, \eqref{lie} and \eqref{g-Kil}, we
establish that $g$ is Killing if and only if
the equalities
\begin{equation}\label{C012}
C_{12}^0=-C_{02}^1=C_{01}^2,  \qquad C_{ij}^i=0
\end{equation}
are valid. Then, using
\eqref{FijkC} and \eqref{titiC}, we obtain the following equalities
\begin{equation}\label{F012C}
-F_{011}=F_{022}=2F_{110}=2F_{101}=-2F_{220}=-2F_{202}=C_{12}^0
\end{equation}
and $\ta=\ta^*=\om=0$. Therefore, by virtue of \eqref{Fi3}, we get
\begin{equation}\label{F810}
F(x, y, z)=F_8(x, y, z)+F_{10}(x, y, z),\qquad 2\lm=-\nu.
\end{equation}
Vice versa, let the latter equalities be satisfied. Then, applying \eqref{Fi3} and \eqref{FijkC} to them,
we deduce \eqref{C012} and \eqref{F012C}. Therefore, $g$ is Killing.
\end{proof}

Similarly, the metric  $\tg$ is Killing if   \eqref{g-Kil} is satisfied for $\tg$, i.e.
$\tg([x,y],z)=\tg(x,[y,z])$ holds.

\begin{thm}
The associated metric $\tg$ of  $(\LLL,\f,\xi,\eta,g)$ is Killing if and only if $(\LLL,\allowbreak{}\f,\allowbreak{}\xi,\allowbreak{}\eta,\allowbreak{}g)$
belongs to the subclass of $\F_8\oplus\F_9\oplus\F_{10}$
determined by the condition $2\lm=\mu=\nu$.
\end{thm}
\begin{proof}
It is analogous to the proof of the previous theorem as equalities \eqref{C012}, \eqref{F012C}, \eqref{F810} are replaced respectively by
\[
C_{12}^0=C_{01}^1=-C_{02}^2,
\]
\[
\begin{array}{l}
F_{120}=F_{102}=-F_{210}=-F_{201}=2F_{110}=2F_{101}\\
\phantom{F_{120}}
=-2F_{220}=-2F_{202}=F_{011}=-F_{022}=C_{12}^0,
\end{array}
\]
\[
F(x, y, z)=F_8(x, y, z)+F_9(x, y, z)+F_{10}(x, y, z),\qquad
2\lm=\mu=\nu.
\]
\end{proof}

The structure $\f$ is called \emph{bi-invariant}, if $\f[x, y] =
[x, \f y]$ is true.
\begin{thm}
The  structure $\f$ of  $(\LLL,\f,\xi,\eta,g)$ is bi-invariant if and only if $(\LLL,\allowbreak{}\f,\allowbreak{}\xi,\allowbreak{}\eta,\allowbreak{}g)$
belongs to the subclass of $\F_4\oplus\F_5\oplus\F_8\oplus\F_{10}$
determined by the condition $2\lm=\nu$.
\end{thm}
\begin{proof}
By a similar way, using \eqref{Fi3}, \eqref{lie}, \eqref{FijkC} and \eqref{titiC}, we establish that the definition condition for a bi-invariant endomorphism $\f$ is satisfied if and only if we have
\[
\begin{array}{l}
F_{120}=F_{102}=F_{210}=F_{201}=C_{01}^1=C_{02}^2=\frac12\ta_0^*,\\
F_{110}=F_{101}=\frac12 C_{12}^0+C_{02}^1=\frac12 C_{12}^0+C_{01}^2=\lm+\frac12\ta_0,\\
F_{220}=F_{202}=-\frac12 C_{12}^0+C_{02}^1=-\frac12 C_{12}^0+C_{01}^2=-\lm+\frac12\ta_0,\\
F_{011}=-F_{022}=C_{12}^0=\nu, \qquad 2\lm=\nu.
\end{array}
\]
Thus, bearing in mind \eqref{Fi3}, we obtain the following
\[
F(x, y, z)=F_4(x, y, z)+F_5(x, y, z)+F_8(x, y, z)+F_{10}(x, y, z),\qquad
2\lm=\nu.
\]
\end{proof}

It is known that $\xi$ is a \emph{Killing vector field} when the Lie derivatite $\LL$ of $g$ along $\xi$ vanishes, \ie $\LL_{\xi}g=0$. According to \cite{ManVes18},
$(\MM,\f,\xi,\eta, g)$ of dimension $(2n+1)$ belongs to $\F_i$ $(i=1,2,3,7,8,10)$ or to their
direct sums.  Then, it is easy to conclude the truthfulness of the following
\begin{thm}
The vestor field $\xi$  of  $(\LLL,\f,\xi,\eta,g)$ is Killing if and only if $(\LLL,\allowbreak{}\f,\allowbreak{}\xi,\allowbreak{}\eta,\allowbreak{}g)$
belongs to $\F_1$, $\F_8$, $\F_{10}$ or to their direct sums.
\end{thm}

\subsection{Curvature properties of the constructed manifolds}

Using \eqref{Kosz} and \thmref{thm-Fi-L}, we obtain the components of $\nabla$ as follows:
\begin{subequations}\label{nabla}
\begin{equation}
\begin{array}{ll}
\F_1: &\n_{E_1}E_1=-\al E_2,\quad  \n_{E_1}E_2=\al E_1, \quad\\
&
\n_{E_2}E_1=\bt E_2, \quad\;\;\; \n_{E_2}E_2=-\bt E_1;
\\[4pt]
\F_4: &\n_{E_1}E_0=-\al E_2,\quad  \n_{E_2}E_0=-\al E_1,\quad\\
&  \n_{E_1}E_2=\n_{E_2}E_1=\al E_0;
\\[4pt]
\F_5: &\n_{E_1}E_0=-\al E_1,\quad \n_{E_2}E_0=-\al E_2\\
&\n_{E_1}E_1=\n_{E_2}E_2=\al E_0;
\\[4pt]
\F_8: &\n_{E_1}E_0=-\al E_2,\quad \n_{E_2}E_0=\al E_1\\
&
\n_{E_1}E_2=-\n_{E_2}E_1=\al E_0;
\\[4pt]
\F_9: &\n_{E_1}E_0=-\al E_1,\quad \n_{E_2}E_0=\al E_2\\
&
\n_{E_1}E_1=-\n_{E_2}E_2=\al E_0;
\\[4pt]
\F_{10}: &\n_{E_0}E_1=-\al E_2,\quad \n_{E_0}E_2=\al E_1;
\end{array}
\end{equation}
\begin{equation}
\begin{array}{ll}
\F_{11}: &\n_{E_0}E_1=\al E_0, \quad \n_{E_0}E_2=\bt E_0\\
&
\n_{E_0}E_0=-\al E_1-\bt E_2,
\end{array}
\end{equation}
\end{subequations}
and the other components are zero.

Then, we have the following

\begin{thm}\label{thm-res}
Let $(\LLL,\f,\xi,\eta,g)$ belong to a basic class $\F_s$, $s\in\{1,4,\allowbreak{}5,8,9,\allowbreak{}10,\allowbreak{}11\}$.
If $s={10}$, $(\LLL,\f,\xi,\eta,g)$ is flat,
whereas if $s\neq 10$, 
$(\LLL,\f,\xi,\eta,g)$ is flat if and only if it is an $\F_0$-manifold. In the latter case the manifold has the following
non-zero components of $R$,
$\rho$, $\rho^*$ and the non-zero values of $\tau$, $\tau^*$,
$k_{ij}$:
\begin{equation*}\label{res}
\begin{array}{ll}
\F_1:\; &R_{1212}=-\rho_{11}=-\rho_{22}=\rho^*_{12}=\rho^*_{21}=-\frac12 \tau=-k_{12} 
=
\al^2+\bt^2;\;  \\[4pt]
\F_4:\;
&R_{0101}=R_{0202}=-R_{1212}=-\frac12
\rho_{00}=-\rho^*_{12}=-\rho^*_{21} \\[4pt]
&\phantom{R_{0101}}=-\frac12 \tau
=-k_{01}=-k_{02}=k_{12}=\al^2;\;\\[4pt]
\F_5:\;
&R_{0101}=R_{0202}=R_{1212}=-\frac12\rho_{00}=-\frac12\rho_{11}
=-\frac12\rho_{22} \\[4pt]
	&\phantom{R_{0101}}=\rho^*_{12}=\rho^*_{21} =-\frac16 \tau=-k_{01}=-k_{02}=-k_{12}=\al^2;\\[4pt]
\F_8:\; &R_{0101}=R_{0202}=-R_{1212}=-\frac12\rho_{00}=-\rho^*_{12}=-\rho^*_{21}  \\[4pt]
&\phantom{R_{0101}}=-\frac12\tau=-k_{01}=-k_{02}=k_{12}=-\al^2;\;  \\[4pt]
\F_9:\; &R_{0101}=R_{0202}=-R_{1212}=-\frac12\rho_{00}=-\rho^*_{12}=-\rho^*_{21} \\[4pt]
&\phantom{R_{0101}}=-\frac12\tau=-k_{01}=-k_{02}=k_{12}=\al^2;\\[4pt]
%
%
\F_{11}:\; &R_{0101}=-\rho_{11}=-k_{01}=\al^2, \quad  \rho_{00}=\frac12 \tau=-(\al^2+\bt^2), \\[4pt]
&R_{0102}=-\rho_{12}=-\frac12 \rho^*_{00}=-\frac12 \tau^*=\al\bt, \\[4pt]
&R_{0202}=-\rho_{22}=-k_{02}=\bt^2.
\end{array}
\end{equation*}
\end{thm}
\begin{proof}
The latter equations are obtained by direct computation of the basic components $R_{ijkl}=R(E_i, E_j, E_k, E_l)$, $\rho_{jk}=\rho(E_j, E_k)$, $\rho^*_{jk}=\rho^*(E_j, E_k)$ and the values of $\tau$, $\tau^*$, $k_{ij}=(E_i, E_j)$, using \thmref{thm-Fi-L} and  \eqref{nabla}.
\end{proof}

Let us remark that in an arbitrary tangent space of $(\LLL,\f,\xi,\eta,g)$ with the basis $\{E_{0},E_{1},E_{2}\}$ defined by
\eqref{strL} and \eqref{gL}, we have two basic $\xi$-sections $\{E_{0},E_{1}\}$,
$\{E_{0},E_{2}\}$ and one basic $\f$-holomorphic section
$\{E_{1},E_{2}\}$.

By virtue to \thmref{thm-res}, we establish the truthfullness of the following
\begin{thm}\label{thm-char}
Let $(\LLL, \f, \xi, \eta, g)$ be a non-flat $\F_s$-manifold, i.e. $s \in \{1,4,5,8,\allowbreak{}9,\allowbreak{}11\}$. Then we have the following characteristics:
\begin{enumerate}
\item The $\F_s$-manifolds ($s=4, 8, 9$) have a curvature tensor of the same form;
\item Every $\F_{11}$-manifold  has the property $R(x,y,\f z,\f w)=0$;
\item Every $\F_{8}$-manifold has a positive scalar
    curvature;
    \item Every $\F_{s}$-manifold ($s=1,4,5,9,11$) has a negative scalar
    curvature;
      \item Every $\F_{s}$-manifold ($s=1,4,5,8,9$) is $*$-scalar flat;
    \item An $\F_{11}$-manifold is $*$-Ricci flat
    if and only if it is $*$-scalar flat;
    \item An $\F_{11}$-manifold is $*$-scalar flat
    if and only if $\al\bt=0$, $(\al,\bt)\neq(0,0)$;
    \item An $\F_{11}$-manifold has a positive (resp., negative) $*$-scalar
    curvature if and only if $\al\bt<0$ (resp., $\al\bt>0$);
     \item  Every $\F_{1}$-manifold  has vanishing sectional curvatures of
    the basic $\xi$-sect\-ions;
       \item  Every $\F_{8}$-manifold  has positive sectional curvatures of
    the basic $\xi$-sect\-ions;
    \item  Every $\F_{s}$-manifold ($s=4,5,9,11$) has negative sectional curvatures
    of the basic $\xi$-sections;
\item  Every $\F_{11}$-manifold has a vanishing scalar
    curvature of the basic $\f$-holomorphic section;
\item Every $\F_{s}$-manifold
    ($s=4,9$) has a positive
    scalar curvature of the basic $\f$-holomorphic section;
    \item Every $\F_{s}$-manifold
    ($s=1,5,8$) has a negative
    scalar curvature of the basic $\f$-holomorphic section.
\end{enumerate}
\end{thm}

Using \thmref{thm-res}, we obtain immediately the following
\begin{cor}\label{cor-Fi3-rho}
The form of the Ricci tensor on $(\LLL, \f, \xi, \eta, g)$  in the corresponding basic class is:
\begin{equation*}\label{Fi3-rho}
\begin{array}{ll}
\F_1: \; \rho=\frac{\tau}{2}\left(g-\eta\otimes\eta\right); \qquad
& \F_s: \; \rho=\tau(\eta\otimes\eta),\quad s\in\{4,8,9\};
\\[4pt]
\F_5: \; \rho=\frac{\tau}{3}g; \quad & \F_{11}: \;
\rho=-\rho(\f\cdot,\f\cdot)+\frac{\tau}{2}g+\tau^*g^*,
\end{array}
\end{equation*}
where $g^*= \tg - \eta\otimes\eta$.
\end{cor}

Bearing in mind \corref{cor-Fi3-rho} and formula \eqref{R3}, we obtain the following
\begin{cor}\label{cor-Fi3-R}
The form of the curvature tensor on  $(\LLL,\f,\xi,\eta,g)$ in the corresponding basic
class is:
\begin{equation*}\label{Fi3-R} %
\begin{array}{llll}
\F_1: & R=-\frac{\tau}{4}\left(g\owedge
g\right)+\frac{\tau}{2}\left(g\owedge
\left(\eta\otimes\eta\right)\right);
\\[4pt]
\F_s: & R=\frac{\tau}{4}\left(g\owedge g\right)
-\tau\left(g\owedge \left(\eta\otimes\eta\right)\right),\qquad s\in\{4,8,9\};
\\[4pt]
\F_5: & R=-\frac{\tau}{12}\left(g\owedge g\right);\\[4pt]
\F_{11}: & R=-\rho \owedge (\eta\otimes\eta).
\end{array}
\end{equation*}
\end{cor}

By virtue to \corref{cor-Fi3-rho}, we obtain the truthfulness of the following
\begin{prop}\label{einstain}
The manifold $(\LLL,\f,\xi,\eta,g)$ is:
\begin{enumerate}
    \item para-$\eta$-Einstein if it belongs to
    $\F_1$, $\F_4$, $\F_5$, $\F_8$, $\F_9$ or to their direct sums;
    \item $\ell_1$-para-$\eta$-Einstein if it belongs to $\F_1$;
    \item $\ell_2$-para-$\eta$-Einstein if it belongs to $\F_4$, $\F_8$, $\F_9$ or to their direct sums;
    \item Einstein if it belongs to $\F_5$;
\end{enumerate}
\end{prop}

\begin{cor}\label{cor-einstain}
The para-Sasakian manifold $(\LLL,\f,\xi,\eta,g)$ is  $\ell_2$-para-$\eta$-Ein\-stein.
\end{cor}

\section{An example of a Lie group as a 3-dimensional manifold of the studied type}
\label{sect-Exms}

In \cite{ManVes18}, an example of an almost paracontact almost paracomplex
Riemannian manifold of arbitrary odd dimension is given.
It is constructed as a family of Lie groups equipped with the studied tensor structure.
Furthermore, certain characteristics of the obtained manifolds are determined. In the present paper, we consider the 3-dimensional case and we find geometrical characteristics in relation the above investigations.

Let $\LLL$ be a 3-dimensional real connected Lie group  and $\mathfrak{l}$ be its associated Lie algebra defined by:
\begin{equation*}\label{com}
    [E_0,E_1]=-a_1E_1-a_{2}E_{2},\quad
    [E_0,E_{2}]=-a_{2}E_1+a_{1}E_{2},\quad
[E_1,E_2]=0,
\end{equation*}
where $a_1, a_{2}$ are real constants and $\{E_0, E_1, E_2\}$ is an $\mathfrak{l}$'s global basis of left invariant vector fields on $\LLL$. An almost paracontact almost paracomplex structure $(\f, \xi, \eta)$ is determined by \eqref{strL} and $g$ is a Riemannian metric defined by \eqref{gL}. Thus, because of \eqref{str}, the induced 3-dimensional
$(\LLL, \f, \xi,\eta, g)$ is an almost paracontact almost paracomplex Riemannian manifold.

Let us remark, the same Lie group  with an appropriate
almost contact structure and a compatible Riemannian
metric is studied in \cite{Ol} as an almost cosymplectic manifold.
The same Lie group is equipped with an almost contact B-metric structure
in \cite{HM} and then certain geometric characteristics for the obtained manifold are found. Further, in \cite{HManMek},
the case of the lowest dimension is considered and some properties of the constructed
manifold are determined.

In \cite{ManVes18}, we get the following components $F_{ijk}=F(E_i,E_j,E_k)$ of  $F$:
\[
\begin{array}{l}
F_{101}=F_{110}=F_{202}=F_{220}=-a_2,\\
F_{102}=F_{120}=-F_{201}=-F_{210}=-a_1,
\end{array}
\]
and the other components of $F$ are zero. Moreover, bearing in mind \eqref{Fi3}, we obtain
$F(x,y,z)=F_4(x,y,z)+F_9(x,y,z)$ for $\mu=-a_1$, $\ta_0=-2a_2$, i.e. $(\LLL,\f, \xi, \eta, g)\in \F_4\oplus\F_9$.
If  $\mu=0$, $\ta_0\neq0$, the manifold belongs to $\F_4$.
Particularly, for $\ta_0=-2$, i.e. $a_1=0$ and $a_2=1$, the obtained manifold is para-Sasakian.
If $\ta_0=0$, $\mu\neq0$, the manifold belongs to $\F_9$.

Now, by virtue of \eqref{Kosz}, for arbitrary $\ta_0$ and $\mu$ we obtain:
\[
\begin{array}{ll}
\n_{E_1}E_0=-\mu E_1-\frac{\ta_0}{2}E_2,\quad &
\n_{E_1}E_2=\n_{E_2}E_1=\frac{\ta_0}{2} E_0,\\[4pt]
\n_{E_2}E_0=-\frac{\ta_0}{2}E_1+\mu E_2,\quad &
\n_{E_1}E_1=-\n_{E_2}E_2=\mu E_0
\end{array}
\]
and the rest are zero.
Then, using the latter equalities, we calculate the basic curvature characteristics and the nonzero of them are the following:
\begin{equation}\label{Rrhotauk-Ex}
\begin{array}{c}
R_{0101}=R_{0202}=-R_{1212}=\mu^2+\frac{\ta_0^2}{4},
\\[4pt]
\rho_{00}=-2(\mu^2+\frac{\ta_0^2}{4}), \qquad
\rho^*_{12}=\rho^*_{21}=-\mu^2-\frac{\ta_0^2}{4},
\\[4pt]
\tau=-2(\mu^2+\frac{\ta_0^2}{4}),
\\[4pt]
k_{01}=k_{02}=-\mu^2-\frac{\ta_0^2}{4},\qquad k_{12}=\mu^2+\frac{\ta_0^2}{4}.
\end{array}
\end{equation}

The latter equalities imply that $(\LLL,\f, \xi, \eta, g)$ has negative  scalar curvature, negative
sectional curvatures of the basic $\xi$-sections and positive sectional curvature
of the basic $\f$-holomorphic section. These results
support \thmref{thm-char} for $\F_4$ and $\F_9$.

According to \eqref{Rrhotauk-Ex}, the form of the Ricci tensor is:
\begin{equation}\label{rho-Ex}
\rho=\tau(\eta\otimes\eta).
\end{equation}
Therefore, $(\LLL,\f, \xi, \eta, g)$ is
an  $\ell_2$-para-$\eta$-Einstein  manifold, which supports \propref{einstain} for $\F_4\oplus\F_9$.

Bearing in mind \eqref{R3} and \eqref{rho-Ex}, we get the following form of the curvature tensor:
\begin{equation*}\label{R-ex}
 R=\frac{\tau}{4}\left(g\owedge g\right)
-\tau\left(g\owedge \left(\eta\otimes\eta\right)\right),
\end{equation*}
which supports \corref{Fi3-R}.


\subsection*{Acknowledgment}
The authors were supported by project MU19-FMI-020 of the Scientific Research Fund,
University of Plovdiv Paisii Hilendarski, Bulgaria.

\end{document}